\documentclass[12pt, a4paper]{amsart}
\usepackage{amsmath}
\usepackage{geometry,amsthm,graphics,tabularx,amssymb,shapepar}
\usepackage{amscd}
\usepackage[all,2cell,dvips]{xy}
\newcommand{\CF}{{\mathcal {F}}}

\newcommand{\CN}{{\mathcal {N}}}
\newcommand{\CO}{{\mathcal {O}}}

\newcommand{\CU}{{\mathcal {U}}}

\newcommand{\Ad}{{\mathrm{Ad}}}

\newcommand{\End}{{\mathrm{End}}}

\newcommand{\GL}{{\mathrm{GL}}}

\newcommand{\Hom}{{\mathrm{Hom}}}

\newcommand{\rank}{{\mathrm{rank}}}

\newcommand{\rk}{{\mathrm{k}}}

\newcommand{\SO}{{\mathrm{SO}}}

\newcommand{\Sp}{{\mathrm{Sp}}}

\newcommand{\tr}{{\mathrm{tr}}}

\newcommand{\vsp}{{\vspace{0.2in}}}

\newcommand{\con}{\textit{C}}

\newcommand{\sdim}{\operatorname{sdim}}

\newcommand{\oH}{\operatorname{H}}
\newcommand{\oJ}{\operatorname{J}}
\newcommand{\oO}{\operatorname{O}}

\newcommand{\oT}{\operatorname{T}}
\newcommand{\oU}{\operatorname{U}}
\newcommand{\oZ}{\operatorname{Z}}
\newcommand{\oD}{\textit{D}}

\newcommand{\h}{\mathfrak h}

\renewcommand{\a}{\mathfrak a}

\renewcommand{\u}{\mathfrak u}

\renewcommand{\j}{\mathfrak j}

\newcommand{\z}{\mathfrak z}
\newcommand{\su}{\mathfrak s \mathfrak u}
\renewcommand{\sl}{\mathfrak s \mathfrak l}

\newcommand{\Z}{\mathbb{Z}}
\newcommand{\C}{\mathbb{C}}
\newcommand{\R}{\mathbb R}

\newcommand{\abs}[1]{\lvert#1\rvert}

\newcommand{\la}{\langle}
\newcommand{\ra}{\rangle}

\newcommand{\be}{\begin {equation}}
\newcommand{\ee}{\end {equation}}
\newcommand{\bee}{\begin {equation*}}
\newcommand{\eee}{\end {equation*}}

\theoremstyle{Theorem}
\newtheorem{introtheorem}{Theorem}

\newtheorem{thm}{Theorem}[section]

\newtheorem{lemt}[thm]{Lemma}

\theoremstyle{Theorem}

\theoremstyle{Theorem}
\newtheorem{prp}{Proposition}[section]
\newtheorem{corp}[prp]{Corollary}
\newtheorem{lemp}[prp]{Lemma}

\theoremstyle{Plain}

\theoremstyle{Plain}
\newtheorem*{thms}{Theorem D$'$}
\newtheorem*{thmss}{Theorem D$''$}

\theoremstyle{Definition}

\newtheorem{dfnt}[thm]{Definition}

\begin{document}

\title[Multiplicity one theorems]{Multiplicity one theorems for Fourier-Jacobi models}

\author{Binyong Sun}
\address{Academy of Mathematics and Systems Science\\
Chinese Academy of Sciences\\
Beijing, 100190, China} \email{sun@math.ac.cn}

%\subjclass[2000]{22E35, 22E46}

%\keywords{Fourier-Jacobi model, classical group, irreducible representation}
%\thanks{Both authors are supported in part by NUS-MOE grant
%R-146-000-102-112.}
%\date{First version on July 30, 2008; Second version on August 5, 2008;
%semi-final version on August 23, 2008}
%\dedicatory{}

\begin{abstract}
For every genuine irreducible admissible smooth representation $\pi$
of the metaplectic group $\widetilde{\Sp}(2n)$ over a p-adic field,
and every smooth oscillator representation $\omega_\psi$ of
$\widetilde{\Sp}(2n)$, we prove that the tensor product $\pi\otimes
\omega_\psi$ is multiplicity free as a smooth representation of the
symplectic group $\Sp(2n)$. Similar results are proved for general
linear groups and unitary groups.
\end{abstract}

\thanks{Supported by NSFC Grants 10801126 and 10931006.}
\maketitle

%\tableofcontents

\section{Introduction}
\label{sec:intro}

Fix a non-archimedean local field $\rk$ of characteristic zero. The
following multiplicity one theorem for general linear groups,
unitary groups and orthogonal groups, which has been expected by J. Bernstein and S. Rallis since
1980's, is established recently by
Aizenbud-Gourevitch-Rallis-Schiffmann in \cite{AGRS}.
\begin{introtheorem}
Let $G$ denote the group $\GL(n)$, $\oU(n)$, or $\oO(n)$, defined
over $\rk$, and let $G'$ denote $\GL(n-1)$, $\oU(n-1)$, or
$\oO(n-1)$, respectively, regarded as a subgroup of $G$ as usual.
Then for any irreducible admissible smooth representation $\pi$ of
$G$, and $\pi'$ of $G'$, one has that
\[
   \dim \Hom_{G'}(\pi\otimes\pi',\C)\leq 1.
\]
\end{introtheorem}
As will be clear later, the groups $\GL(n)$, $\oU(n)$ and $\oO(n)$
are automorphism groups of ``Hermitian modules". Therefore, we
consider Theorem A the multiplicity one theorem in the ``Hermitian
case". It is the first step towards the famous Gross-Prasad
Conjecture (\cite{GP92,GP94,GR06,GGP}).

In \cite{GGP}, W. T. Gan, B. Gross and D. Prasad formulate an analog
of the Gross-Prasad Conjecture in the ``skew-Hermitian case". The
corresponding multiplicity one theorem, whose proof is the main goal
of this paper, is the following:

\begin{introtheorem}\label{intrb}
Let $G$ denote the group $\GL(n)$, $\oU(n)$, or $\Sp(2n)$, defined
over $\rk$, and regarded as a subgroup of the symplectic group
$\Sp(2n)$ as usual. Let $\widetilde G$ be the double cover of $G$
induced by the metaplectic cover $\widetilde{\Sp}(2n)$ of $\Sp(2n)$.
Denote by $\omega_\psi$ the smooth oscillator representation of
$\widetilde{\Sp}(2n)$ corresponding to a non-trivial character
$\psi$ of $\rk$.  Then for any irreducible admissible smooth
representation $\pi$ of $G$, and any genuine irreducible admissible
smooth representation $\pi'$ of $\widetilde G$, one has that
\[
   \dim \Hom_{G}(\pi'\otimes \omega_\psi\otimes \pi,\C)\leq 1.
\]
\end{introtheorem}
Recall that an irreducible admissible smooth representation of
$\widetilde G$ is said to be genuine if it does not descend to a
representation of $G$. For symplectic groups, Theorem B is
conjectured by D. Prasad in \cite[Page 20]{Pr96}. \vsp

The ``$\Hom$"-spaces in Theorem A and Theorem B are extreme cases of
their very important generalizations, namely, Bessel models and
Fourier-Jacobi models, respectively. For definitions of these
models, see \cite[Part 3]{GGP}, for example. As explained in
\cite{GPSR97}, uniqueness of Bessel models is the basic starting
point to study L-functions for orthogonal groups by the
Rankin-Selberg method. Similarly, uniqueness of Fourier-Jacobi
models is basic to study L-functions for symplectic groups and
metaplectic groups (\cite{GJRS09}). The importance of Theorem A and
Theorem B lies in the fact that they imply uniqueness of these
models in general (cf. \cite[Part 3]{GGP}, \cite[Theorem
4.1]{GJRS09} and \cite{JSZ09}). Various special cases of these
uniqueness are obtained in the literature (see \cite{Nov76,GPSR,
BFG92, GRS,BR00,GGP} for example).

\vsp

In order to prove our main results uniformly, we introduce the
following notation. By a commutative involutive algebra (over $\rk$), we mean a
finite product of finite field extensions of $\rk$, equipped with a
$\rk$-algebra involution on it. Let $(A,\tau)$ be a commutative
involutive algebra. Let $E$ be a finitely generated $A$-module. For $\epsilon=\pm 1$,
recall that a $\rk$-bilinear map
\[
  \la\,,\,\ra_E:E\times E\rightarrow A
\]
is called an $\epsilon$-Hermitian form if it satisfies
\[
     \la u,v\ra_E=\epsilon\la v,u\ra_E^\tau, \quad \la au,v\ra_E=a\la u,
     v\ra_E,\quad a\in A,\, u,v\in E.
\]
Assume that $E$ is an $\epsilon$-Hermitian $A$-module, namely it is
equipped with a non-degenerate $\epsilon$-Hermitian form
$\la\,,\,\ra_E$. Denote by $\oU(E)$ the group of all $A$-module
automorphisms of $E$ which preserve the form $\la\,,\,\ra_E$.
Depending on $\epsilon=1$ or $-1$, it is a finite product of general
linear groups, unitary groups, and orthogonal or symplectic groups.

View $A^2$ as a standard hyperbolic plane, i.e., it is equipped with
the $\epsilon$-Hermitian form $\la\,,\ra_{A^2}$ so that both $e_1$
and $e_2$ are isotropic, and that
\[
   \la e_1, e_2 \ra_{A^2}=1,
\]
where $e_1$, $e_2$ is the standard basis of $A^2$. The orthogonal
direct sum $E\oplus A^2$ is again an $\epsilon$-Hermitian
$A$-module.  Note that $\oU(E)$ is identified with the subgroup of
$\oU(E\oplus A^2)$ fixing both $e_1$ and $e_2$. Denote by $\oJ(E)$
the subgroup of $\oU(E\oplus A^2)$ fixing $e_1$.

To be explicit, view $E\oplus A^2$ as the space of column vectors of
$A\oplus E\oplus A$, then $\oJ(E)$ consists of all matrices of the
form
\[
 j(x,u,t):= \left[\begin{array}{ccc}
                  1&-u^\tau& t-\frac{\la u,u\ra_E}{2}\\
                   0 &x& xu\\
                    0&0&1\\
                   \end{array}\right],
\]
where
\[
  x \in \oU(E),\,\,u\in E,\,\, t\in A^{\tau=-\epsilon}:=\{t\in A\mid t^\tau=-\epsilon
  t\},
\]
and $u^\tau$ is the map
\[
   E\rightarrow A, \quad  v\mapsto\la v, u\ra_E.
\]
The unipotent radical of $\oJ(E)$ is
\[
   \oH(E):=\{j(1,u,t)\mid u\in E,\,t\in A^{\tau=-\epsilon}\}=E\times
   A^{\tau=-\epsilon},
\]
with multiplication
\[
   (u,t)(u',t')=(u+u', t+t'+\frac{\la u,u'\ra_E}{2}-\frac{\la
   u',u\ra_E}{2}).
\]
By identifying $j(x,u,t)$ with $(x,(u,t))$, we have
\[
  \oJ(E)=\oU(E)\ltimes \oH(E).
\]

\vsp

By Proposition \ref{repj} of Appendix \ref{appa}, Theorem \ref{intrb} is a consequence of the following theorem
in the skew-Hermitian case, namely when $\epsilon=-1$.

\begin{introtheorem}\label{thmc}
For every irreducible admissible smooth representation $\pi_{\oJ}$
of $\oJ(E)$, and every irreducible admissible smooth representation
$\pi_{\oU}$ of $\oU(E)$, one has that
\[
   \dim \Hom_{\oU(E)}(\pi_{\oJ}\otimes \pi_{\oU},\C)\leq 1.
\]
\end{introtheorem}

\vsp

To prove Theorem C by the method of Gelfand-Kazhdan, we extend
$\oU(E)$ to a larger group, which is denoted by $\breve{\oU}(E)$,
and is defined to be the subgroup of $\GL(E_\rk)\times
\{\pm 1\}$ consists of pairs $(g,\delta)$ such that either
\[
  \delta=1 \quad\textrm{and}\quad \la gu,gv\ra_E=\la u,v\ra_E,\,\,\,  u,v\in E,
\]
or
\[
     \delta=-1 \quad\textrm{and}\quad
     \la gu,gv\ra_E=\la v,u\ra_E,\,\,\,  u,v\in E.
  \]
Here $E_\rk$ is the underlying $\rk$-vector space of $E$. Note that for every element
$(g,\delta)\in \breve{\oU}(E)$, if $\delta=1$, then $g$ is
automatically $A$-linear, and if $\delta=-1$, then $g$ is
$\tau$-conjugate linear. This extended group is first introduced
implicitly by Moeglin-Vigneras-Waldspurger in \cite[Proposition
4.I.2]{MVW87}. It contains $\oU(E)$ as a subgroup of index two. Let
$\breve{\oU}(E)$ act on $\oH(E)$ as group automorphisms by
\begin{equation}\label{gtaction}
     (g,\delta).(u,t):=(gu, \delta t).
\end{equation}
This extends the adjoint action of $\oU(E)$ on $\oH(E)$. The
semidirect product
\[
   \breve{\oJ}(E):=\breve{\oU}(E)\ltimes\oH(E)
\]
contains
\[
   \oJ(E)=\oU(E)\ltimes\oH(E)
\]
as a subgroup of index two.

By Corollary \ref{gkmo2} of Appendix \ref{appb}, Theorem \ref{thmc} is implied by
\begin{introtheorem}
Let $f$ be a generalized function on $\oJ(E)$. If it is invariant
under the adjoint action of $\oU(E)$, i.e.,
\[
    f(g j g^{-1})=f(j),\quad \textrm{for all } g\in\oU(E),
\]
then
\[
    f(\breve{g}j^{-1}\breve{g}^{-1})=f(j),\quad \textrm{for all } \breve{g}\in
    \breve{\oU}(E)\setminus \oU(E).
\]
\end{introtheorem}

The usual notion of generalized functions, as well as the idea of the proof of Theorem D, will be explained in the next section. \vsp

\vsp

The author thanks Gerrit van Dijk, Dipendra Prasad, Lei Zhang and
Chen-Bo Zhu for helpful comments, and thanks the referee for many nice suggestions to improve the paper. He is grateful to Dihua Jiang for teaching him the method of Gelfand-Kazhdan.

\section{Proof of Theorem D}\label{pd}

We first recall some basic notions and facts about distributions and
generalized functions.  By a t.d. space, we mean a topological space
which is Hausdorff, secondly countable, locally compact and totally
disconnected. By a t.d. group, we mean a topological group whose
underlying topological space is a t.d. space. For a t.d space $M$,
denote by $\con^\infty_0(M)$ the space of compactly supported,
locally constant (complex valued) functions on $M$. Denote by
$\oD^{-\infty}(M)$ the space of linear functionals on
$\con^\infty_0(M)$. Such functionals are called distributions on
$M$. If $M$ is furthermore a locally analytic $\rk$-manifold or a
t.d. group (see \cite[Part II]{Sc08} for the notion of locally
analytic manifolds), denote by $\oD^\infty_0(M)$ the space of
compactly supported distributions on $M$ which are locally scalar
multiples of Haar measure. In this case, denote by
$\con^{-\infty}(M)$ the space of linear functionals on
$\oD^\infty_0(M)$. Such functionals are called generalized functions
on $M$.

Let $\varphi: M\rightarrow N$ be a continuous map of t.d. spaces. If
it is proper, then we define the push forward map
\[
  \varphi_*: \oD^{-\infty}(M)\rightarrow \oD^{-\infty}(N)
\]
as usual. Recall that $\varphi$ is said to be proper if
$\varphi^{-1}(C)$ is compact for every compact subset $C$ of $N$.
When $\varphi$ is a closed embedding, $\varphi_*$ identifies
$\oD^{-\infty}(M)$ with distributions in $\oD^{-\infty}(N)$ which
are supported in $\varphi(M)$. If $\varphi$ is not proper, then the
push forward map is still defined, but only for distributions with
compact support.

If  $\varphi: M\rightarrow N$ is a submersion of locally analytic
$\rk$-manifolds, then the push forward map sends $\oD^\infty_0(M)$
into $\oD^\infty_0(N)$, and its transpose defines the pull back map
\[
  \varphi^*: \con^{-\infty}(N)\rightarrow \con^{-\infty}(M).
\]
When $\varphi$ is a surjective submersion, $\varphi^*$ is injective.

If $G$ is an (abstract) group acting continuously on a t.d. space
$M$, then for any group homomorphism $\chi_G:G\rightarrow
\C^\times$, put
\[
   \oD^{-\infty}_{\chi_G}(M):=\{\omega\in \oD^{-\infty}(M)\mid (T_g)_*\,\omega=\chi_G(g)\,\omega,\quad g\in
   G\},
\]
where $T_g:M\rightarrow M$ is the map given by the action of $g\in
G$.  If furthermore $M$ is a locally analytic $\rk$-manifold, and
the action of $G$ on it is also locally analytic, denote by
\[
   \con^{-\infty}_{\chi_G}(M)\subset \con^{-\infty}(M)
\]
the subspace consisting of all $f$ which are $\chi_G$-equivariant,
i.e.,
\[
   f(g.x)=\chi_G(g)f(x),\quad \textrm{for all } g\in
   G,
\]
or to be precise,
\[
  T_g^*(f)=\chi_G(g)f,\quad \textrm{for all } g\in
   G.
\]

\vsp

Now we return to the notation of the Introduction. Recall that $E$
is an $\epsilon$-Hermitian $A$-module. Denote by
\[
  \chi_E: \breve{\oU}(E)\rightarrow \{\pm 1\}, \quad
  (g,\delta)\mapsto \delta
\]
the quadratic character projecting to the second factor. Let
$\breve{\oU}(E)$ act on $\oJ(E)$ by
\begin{equation}\label{actg}
   \breve{g}.j:=\breve{g} \,j^{\,\chi_E(\breve g)}\,\breve{g}^{-1}.
\end{equation}
Then Theorem D is clearly equivalent to
\begin{thms} One has that
$\con^{-\infty}_{\chi_E}(\oJ(E))=0$.
\end{thms}

\vsp

The Lie algebra of (the $\rk$-linear algebraic group)
$\oU(E)$ is
\[
   \u(E):=\{x\in \End_A(E)\mid \la xu,v\ra_E+\la u,xv\ra_E=0,\,u,v\in E\}.
\]
Let $\breve{\oU}(E)$ act on $\u(E)$ and $E$ by
\begin{equation}\label{actioninf}
   \left\{
     \begin{array}{ll}
        (g,\delta).x:=\delta \,gxg^{-1},\quad& x\in \u(E),\medskip\\
        (g,\delta).u:=\delta \,gu,\quad & u\in E,
    \end{array}
  \right.
 \end{equation}
and act on $\u(E)\times E$ diagonally. The linear version of
Theorem D$'$ is
\begin{thmss} One has that
$\con^{-\infty}_{\chi_E}(\u(E)\times E)=0$.
\end{thmss}

A commutative involutive algebra is said to be simple if
it is either a field or a product of two isomorphic fields which are
exchanged by the involution. Write
\[
  (A,\tau)=(A_1,\tau_1)\times (A_2,\tau_2)\times \cdots\times (A_k,\tau_k)
\]
as a product of simple commutative involutive algebras. Then
\[
  E=E_1\times E_2\times \cdots \times E_k,
\]
where $E_j:=A_j\otimes_A E$ is obviously an $\epsilon$-Hermitian $A_j$-module. Note that $E_j$ is free as an $A_j$-module. Put
\[
  \sdim(E):=\dim_{\rk}(E)+\sum_{j=1}^k \max\{\rank_{A_j}(E_j)-1,\,0\}.
\]

Denote by  $\oZ(E)$ the center of $\oU(E)$, and by $\CU_E$ the set of unipotent elements in $\oU(E)$. The following result is proved in Section \ref{reduction}:

\begin{prp}\label{descent11}
Assume that for all commutative involutive algebras $A^\circ$ and
all $\epsilon$-Hermitian $A^\circ$-modules $E^\circ$,
\begin{equation}\label{vanishep1}
  \sdim(E^\circ)<\sdim(E) \quad\textrm{implies}\quad \con^{-\infty}_{\chi_{E^\circ}}(\oJ(E^\circ))=0.
\end{equation}
Then every $f\in \con^{-\infty}_{\chi_{E}}(\oJ(E))$ is supported in
$(\oZ(E)\CU_E)\ltimes \oH(E)$.
\end{prp}

\vsp

With the idea of linearlization by Jaquet-Rallis (\cite{JR96}) in
mind, and based on Proposition \ref{descent11}, we prove the following implication in Section \ref{lin}.
\begin{prp}\label{linear} Theorem D$''$ implies Theorem $D'$.
\end{prp}

\vsp
Now we need to prove Theorem D$''$. Clearly, the case of simple $(A,\tau)$ implies Theorem D$''$ in general. Note that Theorem D$''$ is known when $(A,\tau)$ is simple and $\epsilon=1$. (This is the linear version of the main results (Theorem 2 and Theorem 2$'$) of \cite{AGRS}.) When $\epsilon=-1$, $(A,\tau)$ is simple and $\tau$ is nontrivial, take an nonzero element $c_A\in A$ such that $c_A^\tau=-c_A$, then $(E, c_A\la\,,\,\ra_E)$ is a $-\epsilon$-Hermitian $A$-module, and Theorem D$''$ reduces to the case of $\epsilon=1$. Therefore it remains to prove Theorem D$''$ in the symplectic case, namely, when $A$ is a field, $\tau$ is trivial, and $\epsilon=-1$.

\vsp

Denote by $\z(E)$ the Lie algebra of $\oZ(E)$. It consists of all elements in $\u(E)$ which are scalar multiplications (by certain elements of $A$). Denote by $\CN_E$ the set of all elements in $\u(E)$ which are nilpotent as $\rk$-linear operators on $E$. We first reduce the problem to the null cone as in Proposition \ref{descent11}:

\begin{prp}\label{descent22}
Assume that for all commutative involutive algebras $A^\circ$ and
all $\epsilon$-Hermitian $A^\circ$-modules $E^\circ$,
\begin{equation}\label{vanishep1}
  \sdim(E^\circ)<\sdim(E) \quad\textrm{implies}\quad \con^{-\infty}_{\chi_{E^\circ}}(\u(E^\circ)\times E^\circ)=0.
\end{equation}
Then every $f\in \con^{-\infty}_{\chi_{E}}(\u(E)\times E)$ is supported in
$(\z(E)+\CN_E)\times E$.
\end{prp}

The proof of Proposition \ref{descent22} is similar to that of Proposition \ref{descent11}, and is also carried out in Section \ref{reduction}.

\vsp

Let
\[
  \CN_E=\CN_0\supset \CN_1\supset \cdots \supset \CN_r=\{0\}\supset
  \CN_{r+1}=\emptyset
\]
be a filtration of $\CN_E$ by its closed subsets so that each
difference
\[
  \CO_i:=\CN_i\setminus \CN_{i+1},\quad 0\leq i\leq r,
\]
is a $\breve{\oU}(E)$-orbit (which is also a $\oU(E)$-orbit by \cite[Proposition
4.I.2]{MVW87}).

\begin{prp}\label{indn} Assume that $(A,\tau)$ is simple. Fix $i=0,1,\cdots, r$. If every generalized function in $\con^{-\infty}_{\chi_E}(\u(E)\times E)$ is supported in $(\z(E)+\CN_i)\times E$, then every generalized function  in $\con^{-\infty}_{\chi_E}(\u(E)\times E)$ is supported in $(\z(E)+\CN_{i+1})\times E$.
\end{prp}

Recall that the nilpotent orbit $\CO_i$ is said to be distinguished if one (or every) element of it
commutes with no non-scalar semisimple element in $\u(E)$ (cf. \cite[Section 8.2]{CM}). Use an uncertainty theorem for distributions with supports (Theorem \ref{uncert} of Appendix C), we prove Proposition \ref{indn} for non-distinguished $\CO_i$ in Section \ref{rnd}. We have to go case by case for the proof of Proposition \ref{indn} for distinguished $\CO_i$. It is carried out in the symplectic case in Section \ref{rd}. As explained before, this is the only case which is not done in \cite{AGRS}.

Now we are prepared to prove Theorem D$''$ by induction on $\sdim(E)$. If $\sdim(E)=0$, then $E=0$ and Theorem D$''$ is trivial. Assume that $\sdim(E)>0$ and Theorem D$''$ is proved when $\sdim(E)$ is smaller. Without loss of generality, assume that $(A,\tau)$ is simple.  By Proposition \ref{descent22}, every $f\in \con^{-\infty}_{\chi_{E}}(\u(E)\times E)$ is supported in
$(\z(E)+\CN_E)\times E$, and it has to vanish by Proposition \ref{indn}. This proves Theorem D$''$.

\section{Proofs of Proposition \ref{descent11} and Proposition \ref{descent22}}\label{reduction}

We continue with the notation of the last section. The involution $\tau$ on $A$ extends to an
anti-involution on $\End_A(E)$, which is still denoted by $\tau$, by
requiring that
\[
  \la x^\tau u,v\ra_E=\la u, xv\ra_E, \quad u,v\in E.
\]
Recall that an element $x\in \End_A(E)$ is said to be normal if $x$ commutes with $x^\tau$. For every normal semisimple element $x\in \End_A(E)$, denote by $A_x$ the subalgebra of $\End_A(E)$
generated by $x$, $x^\tau$ and scalar multiplications by $A$. Then $(A_x,\tau)$ is again a
commutative involutive algebra. Write $E_x:=E$, viewed as an $A_x$-module.
\begin{lemp}\label{sdim}
There is a unique $\epsilon$-Hermitian form $\la\,,\,\ra_{E_x}$ on
the $A_x$-module $E_x$ such that
\[
 \dim_\rk(A)\,\, \tr_{A_x/\rk}(\la u,v\ra_{E_x})=\dim_\rk(A_x)\,\, \tr_{A/\rk}(\la u,v\ra_E),\quad u,v\in E.
\]
\end{lemp}

\begin{proof}
The form is determined by requiring
that
\[
   \dim_\rk(A)\,\,  \tr_{A_x/\rk}(a \la u,v\ra_{E_x})=\dim_\rk(A_x)\,\, \tr_{A/\rk}(\la au,v\ra_E),\quad a\in A_x,\,u,v\in E.
\]
\end{proof}

Therefore $E_x$ is an $\epsilon$-Hermitian $A_x$-module. We omit the proof of the following elementary lemma.
\begin{lemp}\label{sdim}
If $x\in \oU(E)\setminus\oZ(E)$ or $\u(E)\setminus \z(E)$, and $x$ is semisimple, then
\[
   \sdim(E_x)<\sdim(E).
\]
\end{lemp}

\vsp

\begin{proof}[Proof of Proposition \ref{descent11}]
Now we come to the proof of Proposition \ref{descent11}. Without loss of generality, assume that $E$ is faithful as an $A$-module. Let $x$ be a semisimple element in $\oU(E)\setminus\oZ(E)$. Note that
$\breve{\oU}(E_x)$ is a subgroup of $\breve{\oU}(E)$. Recall the action of $\breve{\oU}(E)$ on $\oJ(E)$ (and similarly $\breve{\oU}(E_x)$ on $\oJ(E_x)$) from \eqref{actg}.
The homomorphism
\[
    \begin{array}{rcl}
    \xi_x: \oJ(E_x)=\oU(E_x)\ltimes (E_x\times A_x^{\tau=-\epsilon}) &\rightarrow& \oJ(E)=\oU(E)\ltimes (E\times A^{\tau=-\epsilon}) ,\\
             (y,(u,t))&\mapsto & (y,(u,\tr_x(t))
          \end{array}
\]
is $\breve{\oU}(E_x)$-intertwining, where $\tr_x:A_x\rightarrow A$ is the $A$-linear map specified by requiring that
\[
     \dim_\rk(A_x)\,\, \tr_{A/\rk}(\tr_x(t))=  \dim_\rk(A)\,\,\tr_{A_x/\rk}(t),\quad t\in A_x.
\]
Faithfulness of $E$ implies that the map $\tr_x$ is surjective.

For any
\[
  j=(y,h)\in \oJ(E_x)=\oU(E_x)\ltimes \oH(E_x),
\]
denote by $J(j)$  the determinant of the $\rk$-linear map
\[
   1-\Ad_{y^{-1}}: \u(E)/\u(E_x)\rightarrow \u(E)/\u(E_x).
\]
Note that $\Ad_{y}$ preserves a non-degenerate $\rk$-quadratic form
on $\u(E)/\u(E_x)$, which implies that $J$ is
$\breve{\oU}(E_x)$-invariant. Put
\[
  \oJ(E_x)^\circ:=\{j\in \oJ(E_x)\mid J(j)\neq 0\}.
\]
It contains the set $x\CU_{E_x}\ltimes \oH(E_x)$.

One easily checks that the map
\[
 \begin{array}{rcl}
   \rho_x: \breve{\oU}(E)\times \oJ(E_x)^\circ &\rightarrow& \oJ(E),\\
             (\breve{g},j)&\mapsto & \breve{g}.(\xi_x(j))
          \end{array}
\]
is a submersion, and we have a well defined map (cf. \cite[Lemma
2.5]{JSZ})
\begin{equation}\label{reex}
  r_x: \con^{-\infty}_{\chi_{E}}(\oJ(E))\rightarrow  \con^{-\infty}_{\chi_{E_x}}(\oJ(E_x)^\circ),
\end{equation}
which is specified by the rule
\[
   \rho_x^*(f)=\chi _E\otimes  r_x(f),\quad f\in \con^{-\infty}_{\chi_{E}}(\oJ(E)).
\]
Lemma \ref{sdim} and the assumption (\ref{vanishep1}) easily imply
the vanishing of the range space of (\ref{reex}) (cf. \cite[Lemma
2.6]{JSZ}). Thus every $f\in \con^{-\infty}_{\chi_{E}}(\oJ(E))$
vanishes on the image of $\rho_x$. As $x$ is arbitrary, we finish
the proof of Proposition \ref{descent11}.
\end{proof}
\vsp
\begin{proof}[Proof of Proposition \ref{descent22}]
Proposition \ref{descent22} is proved in \cite[Proposition 7.1]{SZ}. We sketch a proof for completeness.   Let $x$ be a semisimple element in $\u(E)\setminus\z(E)$. Recall the action of $\breve{\oU}(E)$ on $\u(E)\times E$ (and similarly $\breve{\oU}(E_x)$ on $\u(E_x)\times E_x$) from \eqref{actioninf}. For any $y\in \u(E_x)$, denote by $J'(y)$  the determinant of the $\rk$-linear map
\[
   [y, \cdot]: \u(E)/\u(E_x)\rightarrow \u(E)/\u(E_x).
\]
Then $J'$ is a $\breve{\oU}(E_x)$-invariant function on $\u(E_x)$. Put
\[
  \u(E_x)^\circ:=\{y\in \u(E_x)\mid J'(y)\neq 0\}.
\]
It contains $x+\CN_{E_x}$.  The map
\[
 \begin{array}{rcl}
    \rho_x': \breve{\oU}(E)\times (\u(E_x)^\circ\times E_x)&\rightarrow& \u(E)\times E,\\
           (\breve g,y,v) &\mapsto &\breve g.(y,v)
   \end{array}
\]
is a submersion, and we finish the proof as that of Proposition \ref{descent11}.
\end{proof}

\section{Proof of Proposition \ref{linear}}\label{lin}

This section is devoted to a proof of Proposition \ref{linear}. So assume throughout this section that
\begin{equation}\label{vinf}
  \con^{-\infty}_{\chi_{E^\circ}}(\u(E^\circ)\times E^\circ)=0
\end{equation}
for all commutative involutive algebras $A^\circ$ and all
$\epsilon$-Hermitian $A^\circ$-modules $E^\circ$.
We are aimed to show that $\con^{-\infty}_{\chi_{E}}(\oJ(E))=0$.

The Lie algebra of $\oH(E)$ is
\[
  \h(E):=E\times A^{\tau=-\epsilon}
\]
with Lie bracket given by
\[
   [(u,t),(u',t')]:=(0,\la u,u'\ra_E-\la u',u\ra_E).
\]
The Lie algebra of $\breve{\oJ}(E)$ is
\[
  \j(E):=\u(E)\ltimes \h(E),
\]
where the semidirect product is defined by the Lie algebra action
\[
      x.(u,t):=(xu, 0).
\]

Let $\breve{\oU}(E)$ act on $\j(E)$ by the differential of its
action on $\oJ(E)$, i.e.,
\[
   \breve{g}.j:=\chi_E(\breve{g})\, \Ad_{\breve g}(j),\quad j\in \j(E).
\]
It is easy to see that as a $\breve{\oU}(E)$-space,
\[
  \j(E)=\u(E)\times E\times A^{\tau=-\epsilon},
\]
where $A^{\tau=-\epsilon}$ carries the trivial
$\breve{\oU}(E)$-action. Therefore the assumption \eqref{vinf} implies that
\begin{equation}\label{vinfj}
  \con^{-\infty}_{\chi_{E}}(\j(E))=0.
\end{equation}

We need the following obvious fact of exponential maps in the theory
of linear algebraic groups.
\begin{lemp}\label{expn}
The set of unipotent elements in  $\oJ(E)$ is $\CU_E\ltimes \oH(E)$,
the set of algebraically nilpotent elements in $\j(E)$ is
$\CN_E\ltimes \h(E)$, and the exponential map is a
$\breve{\oU}(E)$-intertwining homeomorphism from $\CN_E\ltimes
\h(E)$ onto $\CU_E\ltimes \oH(E)$.
\end{lemp}

In all cases that concern us, whenever $M$ is a locally analytic
$\rk$-manifold with a locally analytic $\breve{\oU}(E)$-action,
there is always a canonical choice (up to a scalar) of a positive
smooth invariant measure on $M$. Therefore the space
$\con^{-\infty}_{\chi_{E}}(M)$ is canonically identified with
$\oD^{-\infty}_{\chi_{E}}(M)$. We will use this observation freely.

\begin{lemp}\label{loc222} One has that $\oD^{-\infty}_{\chi_E}(\CU_E\ltimes \oH(E))=0$.
\end{lemp}

\begin{proof}
By \eqref{vinfj}, we have that
\[
\oD^{-\infty}_{\chi_E}(\j(E))=0,
\]
which implies that
\[
   \oD^{-\infty}_{\chi_E}(\CN_E\ltimes \h(E))=0.
 \]
The lemma then follows from Lemma \ref{expn}.
\end{proof}

Recall the following localization principle
which is due to Bernstein. See \cite[section 1.4]{Be84} or
\cite[Corollary 2.1]{AGRS}.

\begin{lemp}\label{localization}
Let $\varphi: M\rightarrow N$ be a continuous map of t.d. spaces,
and let $G$ be a group acting continuously on $M$ preserving the
fibers of $\varphi$.  Then for any group homomorphism
$\chi_G:G\rightarrow \C^\times $, the condition
\[
   \oD_{\chi_G}^{-\infty}(\varphi^{-1}(x))=0 \quad \textrm{for all }x\in N
\]
implies that
\[
  \oD_{\chi_G}^{-\infty}(M)=0.
\]
\end{lemp}

We use Lemma \ref{loc222} and the localization principle to prove the following

\begin{lemp}\label{loc2} One has that $\oD^{-\infty}_{\chi_E}(\oZ(E)\CU_E\ltimes \oH(E))=0$.
 \end{lemp}

\begin{proof}Note that $z\CU_E\ltimes \oH(E)$ is
$\breve{\oU}(E)$-stable for every $z\in \oZ(E)$, and that (by using the trace map) the map
\[
   \begin{array}{rcl}
           \oZ(E)\CU_E\ltimes \oH(E)&\rightarrow & \oZ(E),\\
           (z x, h)&\mapsto& z,
   \end{array}  \qquad (x\in \CU_E)
\]
is a well defined continuous map. By the localization principle, it
suffices to show that
\begin{equation}\label{van22}
   \oD^{-\infty}_{\chi_E}(z\CU_E\ltimes \oH(E))=0,\quad \textrm{for all }z\in\oZ(E).
 \end{equation}
Given $z\in \oZ(E)$, denote by $T_z$ the left multiplication by $z$.
One easily checks that the diagram
\[
   \begin{CD}
      \CU_E\ltimes \oH(E)        @>T_z>>     z\CU_E\ltimes \oH(E) \\
      @V g VV                 @V g VV\\
     \CU_E\ltimes \oH(E)        @>T_z>>     z\CU_E\ltimes \oH(E) \\
   \end{CD}
\]
commutes for all $g\in \oU(E)$, and the diagram
\[
   \begin{CD}
      \CU_E\ltimes \oH(E)        @>T_z>>     z\CU_E\ltimes \oH(E) \\
      @V \breve{g}z VV                 @V \breve{g} VV\\
     \CU_E\ltimes \oH(E)        @>T_z>>     z\CU_E\ltimes \oH(E) \\
   \end{CD}
\]
commutes for all $\breve{g}\in \breve{\oU}(E)\setminus \oU(E)$,
where all vertical arrows are given by the actions of the indicated
elements. Therefore (\ref{van22}) is a consequence of Lemma
\ref{loc222}.
\end{proof}

We now prove Proposition \ref{linear} by induction on $\sdim(E)$. Assume that we have proven the proposition when $\sdim(E)$ is smaller. Then
Proposition \ref{descent11} implies that every $T\in
\oD^{-\infty}_{\chi_{E}}(\oJ(E))$ is supported in
$(\oZ(E)\CU_E)\ltimes \oH(E)$, and then $T=0$ by Lemma \ref{loc2}.
This finishes the proof.

\section{Proof of Proposition \ref{indn}: non-distinguished orbits}\label{rnd}
View $\u(E)$ as a quadratic space over $\rk$ under the trace form
\begin{equation}\label{quadratic}
  \la x,y\ra_{\u(E)}:=\tr_{A/\rk}(\tr(xy)).
\end{equation}
Then we have an orthogonal decomposition
\[
  \u(E)=\z(E)\oplus \su(E),
\]
where $\su(E)$ is the space of trace free elements in $\u(E)$. Recall that $\CO_i=\CN_i\setminus \CN_{i+1}$ is a nilpotent $\breve \oU(E)$-orbit. It is clearly contained in $\su(E)$.

\begin{lemp}\label{metricp}(\cite[Lemma 6.1]{SZ})
If $\CO_i$ is non-distinguished and $o\in \CO_i$, then there is a non-isotropic
vector in $\su(E)$ which is perpendicular to the tangent space
$\oT_o(\CO_i)\subset \su(E)$.
\end{lemp}
\begin{proof}
By definition, $o$ commutes with a nonzero
semisimple element $h\in \su(E)$. Denote by $\a_h$ the center of
$\su(E)^h$ (the centralize of $h$ in $\su(E)$), which is a nonzero non-degenerate subspace of $\su(E)$.

Using the fact that every element of $\a_h$ commutes with $o$, we
see that the tangent space
\[
  \oT_o(\CO_i)=[\u(E),o]=[\su(E), o]
\]
is perpendicular to $\a_h$.

\end{proof}

Recall the action (\ref{actioninf}) of $\breve{\oU}(E)$ on $\u(E)$
and $E$. Write $E':=E$ as an $\epsilon$-Hermitian $A$-module, but
equipped with the action of $\breve{\oU}(E)$ given by
 \[
   (g,\delta).u:=gu.
 \]
Define a
non-degenerate $\breve{\oU}(E)$-invariant bilinear map
\[
   \la\,,\,\ra_\j: (\u(E)\times E)\times (\u(E)\times E')\rightarrow
   \rk
\]
by
\[
   \la (x,u),(x',u')\ra_\j:=\la x, x'\ra_{\u(E)}+\tr_{A/\rk}(\la u,u'\ra_E).
\]
Fix a nontrivial character $\psi$ of $\rk$. As in Appendix C, for
every distribution $T\in \oD^{-\infty}(\u(E)\times E)$, define its
Fourier transform $\widehat{T}\in \con^{-\infty}(\u(E)\times E')$ by
\[
  \widehat{T}(\omega):=T(\hat{\omega}), \quad \omega\in \oD_0^\infty(\u(E)\times
E'),
\]
where $\hat{\omega}\in \con_0^{\infty}(\u(E)\times E)$ is given by
\[
   \hat{\omega}(j):=\int_{\u(E)\times E'} \psi(\la j,j'\ra_\j)\,
   d\omega(j'),\quad j\in \u(E)\times E.
\]

\begin{lemp}\label{funcert}
Assume that $\CO_i$ is non-distinguished. Let $T\in
\oD^{-\infty}(\u(E)\times E)$. If $T$ is supported in $(\z(E)\oplus \CN_i)\times
E$, and its Fourier transform
$\widehat{T}\in\con^{-\infty}(\u(E)\times E')$ is supported in the
null cone
\begin{equation}\label{nulcone}
   \{(z+x,u)\in \u(E)\times E'\mid z\in \z(E),\,x\in \su(E),\,\la x,x\ra_{\u(E)}=0\},
\end{equation}
then $T$ is supported in $(\z(E)\oplus\CN_{i+1})\times E$.
\end{lemp}
\begin{proof}
This is a direct consequence of Lemma \ref{metricp} and Theorem
\ref{uncert} of Appendix C.
\end{proof}

\begin{proof}[Proof of Proposition \ref{indn} for non-distinguished $\CO_i$]
Let $T\in \oD_{\chi_E}^{-\infty}(\u(E)\times E)$. Then by the assumption of  Proposition \ref{indn}
it is supported in $(\z(E)\oplus \CN_i)\times
E$.  It is clear that the Fourier
transform maps $\oD_{\chi_E}^{-\infty}(\u(E)\times E)$ into
$\con^{-\infty}_{\chi_E}(\u(E)\times E')$. By noting that $-1\in
\oU(E)$, we find that the space $\con^{-\infty}_{\chi_E}(\u(E)\times
E')$ is identical to the space
$\con^{-\infty}_{\chi_E}(\u(E)\times E)$. Apply the assumption to
$\widehat{T}$, we find that $\widehat{T}$ is supported in
$(\z(E)\oplus \CN_i)\times
E'$, which is contained in the null cone
(\ref{nulcone}). By Lemma \ref{funcert}, this proves Proposition \ref{indn} in the case that $\CO_i$ is non-distinguished.
\end{proof}

\section{Proof of Proposition \ref{indn}: distinguished orbits}\label{rd}

As explained in Section \ref{pd}, when $\CO_i$ is distinguished, we prove Proposition \ref{indn} only in the symplectic case. So assume that $\epsilon=-1$, $A$ is a field,
and  $\tau$ is trivial. Then $\u(E)$ is a symplectic Lie algebra and $\z(E)=0$. For simplicity of notation and without loss of generality, we further assume that $A=\rk$.

For all
$v\in E$, put
\[
  \phi_v(u):=\la u,v\ra_E \,v,\quad u\in E.
\]
One easily checks that $\phi_v\in \u(E)$. For all $o\in \CO_i$, put
\[
  E(o):=\{v\in E\mid \phi_v\in [\u(E),o]\}.
\]

\begin{lemp}\label{support} If every distribution in $\oD^{-\infty}_{\chi_E}(\u(E)\times
E)$ is supported in $\CN_i\times E$, then the support of every
distribution in $\oD^{-\infty}_{\chi_E}(\u(E)\times E)$ is contained
in
\[
  (\CN_{i+1}\times E) \cup \bigsqcup_{\mathbf{o}\in \CO_i}
  \{o\}\times  E(o).
\]
\end{lemp}

\begin{proof} We follow the method of \cite{AGRS}.
Let $T\in \oD^{-\infty}_{\chi_E}(\u(E)\times E)$ and $(o,v)\in
\CO_i\times E$ be a point in the support of $T$. It suffices to
prove that $v\in E(o)$.

For every $t\in \rk$, define a homeomorphism
\[
 \begin{array}{rcl}
  \eta_t: \u(E)\times E &\rightarrow &\u(E)\times E,\\
                 (x,u)&\mapsto& (x+t \phi_u,u),
 \end{array}
\]
which is checked to be $\breve{\oU}(E)$-intertwining. Therefore
\[
  (\eta_t)_* T\in \oD^{-\infty}_{\chi_E}(\u(E)\times E).
\]
Since $(o,v)$ is in the support of $T$, $\eta_t(o,v)$ is in the
support of $(\eta_t)_* T$. Therefore the assumption implies that
\begin{equation}\label{etat}
  \eta_t(o,v)=(o+t\phi_v, v)\in \CN_i\times E.
\end{equation}
As $\CO_i$ is open in $\CN_i$, (\ref{etat}) implies that
\[
  \phi_v\in \oT_o(\CO_i)=[\u(E),o].
\]

\end{proof}

Fix an element $\mathbf e\in \CO_i$. Extend it to a standard triple
$\mathbf{h},\mathbf{e},\mathbf{f}$ in $\u(E)$, i.e., the
$\rk$-linear map from $\sl_2(\rk)$ to $\u(E)$ specified by
\[
  \left[
                   \begin{array}{cc} 1&0\\ 0&-1\\
                   \end{array}
  \right]
   \mapsto \mathbf{h},\quad
 \left[
                   \begin{array}{cc} 0&1\\ 0&0\\
                   \end{array}
  \right]
   \mapsto \mathbf{e},\quad
   \left[
                   \begin{array}{cc} 0&0\\ 1&0\\
                   \end{array}
  \right]
   \mapsto \mathbf{f},
\]
is a Lie algebra homomorphism. Existence of such an extension is
known as Jacobson-Morozov Theorem. Using this homomorphism, we view
$E$ as an $\sl_2(\rk)$-module with an invariant symplectic form. In
the remaining part of this section assume that $\CO_i$ is
distinguished. By the classification of distinguished nilpotent
orbits (\cite[Theorem 8.2.14]{CM}), we know that $E$ has an
orthogonal decomposition
\begin{equation}\label{decomes}
  E=E_1\oplus E_2\oplus\cdots\oplus E_s,\quad s\geq 0,
\end{equation}
where all $E_j$'s are irreducible $\sl_2(\rk)$-submodules, with
pairwise different even dimensions. Denote by $E^+$ and $E^-$ the
subspaces of $E$ spanned by eigenvectors of $\mathbf h$ with
positive and negative eigenvalues, respectively. Then
\[
  E=E^+\oplus E^-
\]
is a complete polarization of $E$.

\begin{lemp}\label{subs1} (cf.
\cite[Lemma 4.1]{SZ})
One has that $E(\mathbf e)=E^+$.
\end{lemp}
\begin{proof}
Recall that $\u(E)$ is a quadratic space over $\rk$ under the trace form \eqref{quadratic}. For every $v\in E$, we have that $v\in E(\mathbf e)$ if and only if
\begin{eqnarray*}
&&\phi_{v}\in [\u(E), \mathbf e] \Leftrightarrow \phi_{v}\perp [\u(E), \mathbf e]^\perp\\
&&\phantom { \phi_{v}\in [\u(E), \mathbf e] } \Leftrightarrow  \phi_{v}\perp \u(E)^{\mathbf e} \,\,\textrm{ (the centralizer of $\mathbf e$ in $\u(E)$)}\\
&&\phantom { \phi_{v}\in [\u(E), \mathbf e] } \Leftrightarrow  \la x v,  v\ra_E=0\,\, \textrm{ for all } x\in \u(E)^{\mathbf e}.
\end{eqnarray*}
Thus if $v\in E(\mathbf e)$, then we have
\[
 \la \mathbf{e}_j^{2k+1}v,   v\ra_E=0\quad \textrm{ for all $1\leq j\leq s$ and $k\geq 0$},
\]
where $\mathbf{e}_j$ is  the restriction of $\mathbf e$ to
$E_j$ under the decomposition \eqref{decomes}. Therefore $v\in E^{+}$.

On the other hand, every element $x\in \u(E)^{\mathbf e}$ stabilizes $E^{+}$. Therefore $v\in E^{+}$ implies that $\la x v, v\ra_E=0$. This finishes the proof.
\end{proof}

\vsp

Fix a Haar  measure $du'$ on $E'$. For any t.d. space $M$, we define
the partial Fourier transform
\[
   \CF_E: \oD^{-\infty}(M\times E)\rightarrow \oD^{-\infty}(M\times E')
\]
by
\[
  \CF_E(T)(\varphi_M\otimes \varphi'):=T(\varphi_M\otimes \hat{\varphi'}), \quad
  \varphi_M\in \con^\infty_0(M), \,\varphi'\in \con^\infty_0(E'),
\]
where $\hat{\varphi'}\in \con^\infty_0(E)$ is given by
\[
  \hat{\varphi'}(u):=\int_{E'} \psi(\la u,u'\ra_E)\,\varphi'(u')
   \,du'.
\]

For every $o\in \CO_i$, write $E'(o):=E(o)$, viewed as a subset of
$E'$. Lemma \ref{subs1} implies the following
\begin{lemp}\label{haarm}
Let $T\in \oD^{-\infty}(E)$. If $T$ is supported in $E(\mathbf e)$,
and $\CF_E(T)\in \oD^{-\infty}(E')$ is supported in $E'(\mathbf e)$,
then $T$ is a scalar multiple of a Haar  measure of $E^+$.
\end{lemp}
\begin{proof}Since $\CF_E(T)$ is supported in $E'(\mathbf e)=E^+$, $T$ is
invariant under translations by elements of
\[
  \{u\in E\mid \la u, u'\ra_E=0, \quad u'\in E'(\mathbf e)\}=E^+.
\]
\end{proof}

Denote by $\breve{\oU}(E,\mathbf e)$ the stabilizer of $\mathbf e\in
\CO_i$ in $\breve{\oU}(E)$, and by $\chi_{E,\mathbf e}$ the
restriction of $\chi_E$ to $\breve{\oU}(E,\mathbf e)$.
\begin{lemp}\label{van111}
Let $T\in \oD^{-\infty}_{\chi_{E,\mathbf e}}(E)$. If $T$ is
supported in $E(\mathbf e)$, and $\CF_E(T)$ is supported in
$E'(\mathbf e)$, then $T=0$.
\end{lemp}
\begin{proof}
By Lemma \ref{haarm}, $T$ is a scalar multiple of a Haar  measure of
$E^+$. Note that all eigenvalues of $\mathbf h$ on $E$ are odd
integers. Let $g:E\rightarrow E$ be the linear map which is the
scalar multiplication by $(-1)^n$ on the $\mathbf h$-eigenspace with
eigenvalue $2n+1$, $n\in \Z$. It is clear that $(g,-1)\in
\breve{\oU}(E,\mathbf e)$, and leaves the Haar  measure of
$E^+$-invariant. This finishes the proof.
\end{proof}

\vsp Fix a positive $\breve{\oU}(E)$-invariant measure $do$ on
$\CO_i$ (which always exists), and a Haar  measure $d\breve{g}$ on
$\breve{\oU}(E)$. Define a submersion
\[
  \begin{array}{rcl}
  \rho_\mathbf{e}: \breve{\oU}(E)\times E&\rightarrow&\CO_i\times E,\\
      (\breve{g}, v)&\mapsto &\breve{g}.(\mathbf e,v),
  \end{array}
\]
and define the pull back
\[
     \begin{array}{rcl}
   \rho_\mathbf{e}^*: \oD^{-\infty}(\CO_i\times E) &\rightarrow& \oD^{-\infty}(\breve{\oU}(E)\times
   E),\\
           f\,do\otimes du&\mapsto & \rho_{\mathbf e}^*(f)\,d\breve{g}\otimes
           du,
   \end{array}
\]
where $du$ is any Haar  mesure on $E$, $f\in
\con^{-\infty}(\CO_i\times E)$, and $\rho_{\mathbf e}^*(f)$ is the
usual pull back of a generalized function.  By Frobenius reciprocity (cf. \cite[Section 1.5]{Be84}), there is a well defined
linear isomorphism
\begin{equation}\label{rmathbfe}
  r_\mathbf{e}: \oD^{-\infty}_{\chi_E}(\CO_i\times
  E)\stackrel{\sim}{\rightarrow}
  \oD^{-\infty}_{\chi_{E,\mathbf{e}}}(E),
\end{equation}
specified by
\[
  \rho_\mathbf{e}^*(T)=\chi_E\,d\breve{g}\otimes  r_\mathbf{e}(T),\quad T\in \oD^{-\infty}_{\chi_E}(\CO_i\times
  E).
\]
Similarly, by using the action of $\breve{\oU}(E)$ on $\CO_i\times
E'$, we define a map
\[
     \begin{array}{rcl}
   {\rho'}_\mathbf{e}^*: \oD^{-\infty}(\CO_i\times E') &\rightarrow& \oD^{-\infty}(\breve{\oU}(E)\times
   E'),\\
        \end{array}
\]
and a linear isomorphism
\[
  r'_\mathbf{e}: \oD^{-\infty}_{\chi_E}(\CO_i\times E')\stackrel{\sim}{\rightarrow}
  \oD^{-\infty}_{\chi_{E,\mathbf{e}}}(E').
\]

The routine verification of the following lemma is left to the
reader.
\begin{lemp}\label{cdiag}
The diagram
\[
   \begin{CD}
     \oD^{-\infty}_{\chi_E}(\CO_i\times E)     @>r_{\mathbf e}>>     \oD^{-\infty}_{\chi_{E,\mathbf e}}(E) \\
      @V \CF_E VV                 @V \CF_E  VV\\
     \oD^{-\infty}_{\chi_E}(\CO_i\times E')   @>r'_{\mathbf e}>>     \oD^{-\infty}_{\chi_{E,\mathbf e}}(E')\\
   \end{CD}
\]
commutes.
\end{lemp}

\vsp

Now we are ready to prove Proposition \ref{indn} for distinguished $\CO_i$ (in the symplectic case). Let $T\in \oD^{-\infty}_{\chi_E}(\CN_i\times E)$.
Then Lemma \ref{support} implies that
\[
  r_\mathbf{e}(T|_{\CO_i\times E})\in \oD^{-\infty}_{\chi_{E,\mathbf
  e}}(E)\quad\textrm{is supported in $E(\mathbf e)$.}
\]
Similarly, since $\CF_E(T)\in\oD^{-\infty}_{\chi_E}(\CN_i\times
E')$, we have
\[
  r'_\mathbf{e}(\CF_E(T)|_{\CO_i\times E'})\in \oD^{-\infty}_{\chi_{E,\mathbf
  e}}(E') \quad\textrm{is supported in $E'(\mathbf e)$.}
\]
Lemma \ref{van111} and Lemma \ref{cdiag} implies that
\[
  r_\mathbf{e}(T|_{\CO_i\times E})=0,
\]
which implies that $T|_{\CO_i\times E}=0$. This finishes the proof.

\vsp

\appendix

\section{Some Mackey theory}\label{appa}

Let $\rk$ be a non-archimedean local field of characteristic zero, and let
\[
 (A,\tau)=\left\{
                  \begin{array}{l}
                 (\rk\times \rk, \textrm{the nontrivial
                 automophism}),\\
                  (\textrm{a quadratic field extension of $\rk$}, \textrm{the nontrivial
                 automophism}),\, \textrm{or}\\
                    (\rk, \textrm{the trivial
                 automophism}).
                 \end{array}
            \right.
\]
It is a commutative involutive algebra over $\rk$. Let $\epsilon=-1$, and let $(E, \la\,,\,\ra_E)$ be an $\epsilon$-Hermitian $A$-module. Then the group $\oU(E)$ is $\GL(n)$, $\oU(n)$, or $\Sp(2n)$, respectively,
with $2n=\dim_\rk (E)$. Write $E_\rk:=E$, viewed as a symplectic
$\rk$-vector space under the form
\[
  \la u,v\ra_{E_\rk}:=\frac{1}{\dim_{\rk}(A)}\tr_{A/\rk}(\la u,v\ra_E).
\]

Recall the Heisenberg group
\[
  \oH(E)=E\times  A^{\tau=-\epsilon}=E_\rk\times \rk=\oH(E_\rk)
\]
with group multiplication
\[
   (u,t)(u',t')=(u+u', t+t'+\frac{\la u,u'\ra_E}{2}-\frac{\la
   u',u\ra_E}{2})=(u+u', t+t'+\la u,u'\ra_{E_\rk}).
\]
Denote by $\widetilde{\Sp}(E_\rk)$ the metaplectic cover of the
symplectic group $\Sp(E_\rk)$. It induces a double cover
$\widetilde{\oU}(E)$ of $\oU(E)\subset \Sp(E_\rk)$.  For any
non-trivial character $\psi$ of $\rk$, denote by $\omega_\psi$ the
corresponding smooth oscillator representation of
\begin{equation}\label{jacobisp}
  \widetilde{\Sp}(E_\rk)\ltimes \oH(E_\rk).
\end{equation}
Up to isomorphism, this is the only genuine smooth representation
which, as a representation of $\oH(E_\rk)$, is irreducible and has
central character $\psi$. We regard $\omega_\psi$ as a representation of $\widetilde{\oU}(E)\ltimes \oH(E)$, as the later is a subgroup of \eqref{jacobisp}.

The following Mackey-theoretic result is known
(cf. \cite[Page 222]{AP06}). We provide a proof for the sake of
completeness.

\begin{prp}\label{repj}
With the notation as above, for every genuine irreducible admissible
smooth representation $\pi_{\widetilde{\oU}}$ of
$\widetilde{\oU}(E)$, the tensor product
$\pi_{\widetilde{\oU}}\otimes \omega_\psi$ is an irreducible
admissible smooth representation of $\oJ(E)=\oU(E)\ltimes \oH(E)$.
\end{prp}
\begin{proof}
The smoothness and admissibility are clear. We prove that
$\pi_{\widetilde{\oU}}\otimes \omega_\psi$ is irreducible as a
smooth representation of $\oU(E)\ltimes \oH(E)$. The space
\[
  \oH_0:=\Hom_{\oH(E)}(\omega_\psi, \pi_{\widetilde{\oU}}\otimes
  \omega_\psi)
\]
is a smooth representation of $\widetilde{\oU}(E)$ under the action
\[
  (\tilde{g}.\phi)(v):=g.(\phi(\tilde{g}^{-1}.v)),
\]
where
\[
 \quad
  \tilde{g}\in\widetilde{\oU}(E), \,\phi\in  \oH_0, \, v\in \omega_\psi,
\]
and $g$ is the image of $\tilde{g}$ under the quotient map
$\widetilde{\oU}(E)\rightarrow \oU(E)$. Let $\pi_{\oJ}$ be a nonzero
$\oU(E)\ltimes \oH(E)$-subrepresentation of
$\pi_{\widetilde{\oU}}\otimes
  \omega_\psi$, then
\[
  \Hom_{\oH(E)}(\omega_\psi, \pi_{\oJ})
\]
is a nonzero $\widetilde{\oU}(E)$-subrepresentation of $\oH_0$.
Since the linear map
 \[
   \pi_{\widetilde{\oU}} \rightarrow \oH_0,\quad   v\mapsto  v\otimes (\,\cdot\,)
\]
is bijective and $\widetilde{\oU}(E)$-intertwining, $\oH_0$ is
irreducible. Therefore
\[
  \Hom_{\oH(E)}(\omega_\psi, \pi_{\oJ})=\oH_0,
\]
and consequently, $\pi_{\oJ}=\pi_{\widetilde{\oU}}\otimes
\omega_\psi$.

\end{proof}

\section{The Gelfand-Kazhdan criterion for multiplicity one pairs}\label{appb}

Recall the notion of t.d. groups from Section \ref{reduction}. The following result is a form of the Gelfand-Kazhdan criterion for multiplicity one pairs.

\begin{prp}\label{gkmo}
Let $G$ be a t.d. group with a closed subgroup $S$. Let $\sigma$ be
a continuous anti-automorphism of $G$ such that $\sigma(S)=S$.
Assume that for every generalized function $f$ on $G$ or on $S$, the
condition
\[
   f(sxs^{-1})=f(x)\quad \textrm{ for all } s\in S
\]
implies that
\[
  f(x^\sigma)=f(x).
\]
Then for all irreducible admissible smooth representation $\pi_G$ of
$G$, and $\pi_S$ of $S$, one has that
\[
   \dim \Hom_{S}(\pi_G\otimes \pi_S,\C)\leq 1.
\]
\end{prp}
\begin{proof}
This is proved for real reductive groups in \cite[Corollary 2.5]{SZ08}. The same proof works here. We sketch a proof for convenience of the reader.

Denote by $\Delta(S)$ the diagonal subgroup $S$ of $G\times S$. The
assumption on $G$ implies that every bi-$\Delta(S)$ invariant
generalized function on $G\times S$ is $\sigma\times
\sigma$-invariant. Then the usual Gelfand-Kazhdan criterion (cf.
\cite[Theorem 2.3]{SZ08}) implies that
\begin{equation}\label{multiplication}
  \dim \Hom_S(\pi_G \otimes \pi_S,\C)\cdot \dim \Hom_S(\pi_G^\vee \otimes\pi_S^\vee,\C)\leq
  1.
\end{equation}
Here and henceforth, ``$\,^\vee$" stands for the contragredient of
an admissible smooth representation.

Denote by $\sigma'$ the automorphism $g\mapsto \sigma(g^{-1})$. By
considering characters of irreducible admissible smooth representations (which are
conjugation invariant generalized functions on the groups), the
assumption implies that
\[
  \pi_G^\vee\cong \pi_G^{\sigma'}\quad\textrm{and}\quad \pi_S^\vee\cong
  \pi_S^{\sigma'}.
\]
Here $\pi_G^{\sigma'}$ is the representation of $G$ which has the
same underlying space as that of $\pi_G$, and whose action is given
by $g\mapsto \pi_G(\sigma'(g))$. The representation
$\pi_S^{\sigma'}$ is defined similarly. Therefore the two
factors in \eqref{multiplication} are equal to each other, and
consequently,
\[
  \dim \Hom_S(\pi_G \otimes \pi_S,\C)\leq 1.
\]

\end{proof}

For unimodular groups, we have
\begin{corp}\label{gkmo2}
Let $G$ be a unimodular t.d. group with a  unimodular closed subgroup $S$. Let $\sigma$ be
a continuous anti-automorphism of $G$ such that $\sigma(S)=S$.
Assume that for every generalized function $f$ on $G$, the
condition
\[
   f(sxs^{-1})=f(x)\quad \textrm{ for all } s\in S
\]
implies that
\[
  f(x^\sigma)=f(x).
\]
Then for all irreducible admissible smooth representation $\pi_G$ of
$G$, and $\pi_S$ of $S$, one has that
\[
   \dim \Hom_{S}(\pi_G\otimes \pi_S,\C)\leq 1.
\]
\end{corp}
\begin{proof}
When both $G$ and $S$ are unimodular, the assumption of the corollary implies the assumption of Proposition \ref{gkmo}. Therefore the corollary is a consequence of Proposition \ref{gkmo}.
\end{proof}

\section{An uncertainty theorem for distributions with supports}\label{appc}

Let $\rk$ be a non-archimedean local field of characteristic zero.
Fix a non-trivial character $\psi$ of $\rk$. Let $E$ and $F$ be two
finite-dimensional $\rk$-vector spaces which are dual to each other,
i.e., a non-degenerate bilinear map
\[
  \la \,,\,\ra: E\times F\rightarrow \rk
\]
is given. The Fourier transform
\[
  \begin{array}{rcl}
    \oD_0^\infty(F)&\rightarrow &\con_0^\infty(E)\\
        \omega&\mapsto &\hat{\omega}
    \end{array}
\]
is the linear isomorphism given by
\[
   \hat{\omega}(x):=\int_{F} \psi(\la x,y\ra)\,
   d\omega(y),\quad x\in E.
\]
For every $T\in \oD^{-\infty}(E)$, its Fourier transform $\widehat
T\in \con^{-\infty}(F)$ is given by
\[
  \widehat{T}(\omega):=T(\hat{\omega}), \quad \omega\in \oD_0^\infty(F).
\]

For every subset $X$ of $E$, a point $x\in X$ is said to be regular
if there is an open neighborhood $U$ of $x$ in $E$ such that $U\cap
X$ is a closed locally analytic submanifold of $U$. In this case,
the tangent space $\oT_x(X)\subset E$ is defined as usual. We define
the conormal space to be
\[
  \operatorname{N}^*_x(X):=\{v\in F\mid \la u,v\ra=0, \quad u\in \oT_x(X)\}.
\]

The uncertainty principle says that a distribution and its Fourier
transform can not be simultaneously arbitrarily concentrated. The
purpose of this appendix is to prove the following theorem, which is
a form of the uncertainty principle.

\begin{thm}\label{uncert}
Let $x$ be a regular point in a close subset $X$ of $E$. Let
$f:F\rightarrow \rk$ be a polynomial function of degree $d\geq 1$, and denote by $f_d$ its homogeneous component of degree $d$. Let $T\in \oD^{-\infty}(E)$ be a distribution supported in $X$, with
its Fourier transform $\widehat T\in \con^{-\infty}(F)$ supported in
the zero locus of $f$.  If $f_d$ take nonzero values at some points
of  $\operatorname{N}^*_x(X)$, then $T$ vanishes on some open
neighborhood of $x$ in $E$.
\end{thm}

\noindent {\bf Remark}: The archimedean analog of Theorem
\ref{uncert} also holds. This is a direct consequence of \cite[Lemma
2.2]{JSZ}.

\vsp

Fix a non-archimedean multiplicative norm
\[
\abs{\,\cdot\,}_\rk:\rk\rightarrow [0,+\infty)
\]
which defines the topology of $\rk$. Also fix a non-archimedean norm
(multiplicative with respect to $\abs{\,\cdot\,}_\rk$)
\[
\abs{\,\cdot\,}_{F}:F\rightarrow [0,+\infty),
\]
which automatically defines the topology of $F$.

Let $f$ and $f_d$ be as in Theorem \ref{uncert}, and denote by
$Z_f\subset F$ the zero locus of $f$. Write $f_0:=f-f_d$, which is a
polynomial function of degree $\leq d-1$. Then
\begin{equation}\label{od}
   \abs{f_0(y)}_\rk =o(\abs{y}_{F}^d), \quad
  \textrm{as }\abs{y}_F\rightarrow +\infty.
\end{equation}

\begin{lemt}\label{intemp}
Let $V_1$ and $V_2$ be two compact open subsets of $F$. If $f_d$ has
no zero in $V_2$, then
\[
  (Z_{f}+V_1)\cap \lambda V_2=\emptyset
\]
for all $\lambda\in \rk^\times$ with $\abs{\lambda}_\rk$
sufficiently large.
\end{lemt}
\begin{proof}
Take a positive number $c$ so that
\begin{equation}\label{e1}
  \abs{f_d(y)}_\rk\geq c
  \abs{y}_{F}^d,\quad\textrm{for all }y\in \rk^\times
  V_2.
\end{equation}

It is easy to see that
\begin{equation}\label{e2}
  \max_{v\in V_1}{\abs{f_d(y+v)-f_d(y)}}_\rk=o(\abs{y}_{F}^d), \quad
  \textrm{as }\abs{y}_F\rightarrow +\infty.
\end{equation}
If $y\in Z_f$, then
\begin{eqnarray}\label{e3}
  &&\phantom{=}\abs{f_d(y+v)}_\rk\\ \nonumber
  &&\leq  \max\{\,\abs{f_d(y)}_\rk,\,\abs{f_d(y+v)-f_d(y)}_\rk\,\}\\ \nonumber
  &&=\max\{\,\abs{f_0(y)}_\rk,\,\abs{f_d(y+v)-f_d(y)}_\rk\,\}.
\end{eqnarray}
The inequalities (\ref{od}), (\ref{e2}) and (\ref{e3}) implies that
\begin{equation}\label{e4}
  \abs{f_d(y)}_\rk=o(\abs{y}_{F}^d), \quad
  \textrm{as }y\in Z_{f}+V_1,\, \textrm{ and } \abs{y}_F\rightarrow +\infty.
\end{equation}
The lemma then follows by comparing (\ref{e1}) and (\ref{e4}).

\end{proof}

Recall the following
\begin{dfnt} (cf. \cite[Section 2]{He85})
A distribution  $T\in \oD^{-\infty}(E)$ is said to be smooth at a
point $(x,y)\in E\times F$ if there is a compact open neighborhood
$U$ of $x$, and a compact open neighborhood $V$ of $y$ such that the
Fourier transform $\widehat{1_U T}$ vanishes on $\lambda V$ for all
$\lambda\in \rk^\times$ with $\abs{\lambda}_\rk$ sufficiently large.
Here $1_U$ stands for the characteristic function of $U$. The wave
front set of $T$ at $x\in E$ is defined to be
\[
  \mathrm{WF}_x(T):=\{y\in F\mid T\textrm{ is not smooth at }
  (x,y)\}.
\]

\end{dfnt}
Clearly, the wave front set $\mathrm{WF}_x(T)$ is closed in $F$ and
is stable under multiplications by $\rk^\times$.

\begin{lemt}\label{wavef}
If the Fourier transform $\widehat T$ of a distribution $T\in
\oD^{-\infty}(E)$ is supported in $Z_f$, then for every $x\in E$,
the wave front set $\mathrm{WF}_x(T)$ is contained in the zero locus
of $f_d$.
\end{lemt}
\begin{proof}
Let $y\in F$ be a vector so that $f_d(y)\neq 0$. We need to show
that $T$ is smooth at $(x,y)$. Take an arbitrary compact open
neighborhood $U$ of $x$, and an arbitrary compact open neighborhood
$V$ of $y$ so that $f_d$ has no zero in $V$. We claim that
$\widehat{1_U T}$ vanishes on $\lambda V$ for all $\lambda\in
\rk^\times$ with $\abs{\lambda}_\rk$ sufficiently large. The lemma
is a consequence of this claim.

Note that $\widehat{1_U T}$ is a finite linear combination of
generalized functions of the form
\[
  (1_{V_1}\,dy)*\widehat{T}, \quad \textrm{$V_1$ is a
  compact open subset of $F$}.
\]
Here $dy$ is a fixed Haar  measure on $F$. The support of the
convolution $(1_{V_1}\,dy)*\widehat{T}$ is contained in $Z_f+V_1$.
Therefore the claim follows from Lemma \ref{intemp}.
\end{proof}

\begin{lemt}\label{trans}
If a distribution $T\in \oD^{-\infty}(E)$ is supported in a closed
subset $X$ of $E$, and $x\in X$ is a regular point, then the wave
front set $\mathrm{WF}_x(T)$ is invariant under translations by
elements of $\mathrm{N}_x^*(X)\subset F$.
\end{lemt}
\begin{proof}
This is proved in \cite[Therorem 4.1.2]{Ai}. We indicate the main
steps.

Step 1. When we replace a distribution by a translation of it, the
wave front set does not change. Therefore we may assume that $x=0$.

Step 2. Let $\varphi: E\rightarrow E$ be a locally analytic
diffeomorphism which sends $0$ to $0$ and induces the identity map
on the tangent space at $0$. When we replace $T$ by its pushing
forward $\varphi_*(T)$, the wave front set $\operatorname{WF}_0(T)$
does not change. Therefore we may assume that $X\cap U=E_0\cap U$,
for some subspace $E_0$ of $E$, and some open neighborhood $U$ of
$0$.

Step 3. When we replace $T$ by a distribution which coincides with
$T$ on an open neighborhood of $0$, the wave front set
$\operatorname{WF}_0(T)$ does not change. Therefore we may assume
that $X=E_0$.

Step 4. Assume that $T$ is supported in $E_0$. Then $\widehat T$ is
invariant under translations by
\[
 \mathrm{N}^*_x(X)=E_0^\perp:=\{v\in F\mid \la u,v\ra=0, \quad u\in E_0\},
\]
which implies that the same holds for $\operatorname{WF}_0(T)$.
\end{proof}

Now we are ready to prove Theorem \ref{uncert}. It is clear that $T$
vanishes on some open neighborhood of $x$ if and only if $0\notin
\operatorname{WF}_x(T)$. If $0\in \operatorname{WF}_x(T)$, then
Lemma \ref{trans} implies that $\mathrm{N}_x^*(X)\subset
\operatorname{WF}_x(T)$. Now Lemma \ref{wavef} further implies that
$\mathrm{N}_x^*(X)$ is contained in the zero locus of $f_d$, which
contradicts the assumption of the theorem.

\end{document}